\documentclass[11pt]{article}
\usepackage{amssymb,amsthm,amsmath,mathtools}
\usepackage[dvipsnames]{xcolor}
\usepackage{graphicx}
\numberwithin{equation}{section}
\theoremstyle{plain}
\newtheorem{theorem}{Theorem}[section]
\newtheorem{lemma}[theorem]{Lemma}
\newtheorem{proposition}[theorem]{Proposition}
\newtheorem{corollary}[theorem]{Corollary}
\newtheorem{definition}[theorem]{Definition}
\newtheorem{remark}[theorem]{Remark}
\theoremstyle{definition}

\usepackage{vmargin}
\setmarginsrb{2cm}{1cm}{2cm}{3cm}{1cm}{1cm}{2cm}{1cm}
\usepackage[numbers]{natbib}

\begin{document}
\title{On the Homogenization of Norm and Time Optimal Controls for  \\
Parabolic Equations with Oscillating Coefficients}
\author{Jon Asier Bárcena-Petisco\thanks{Department of Mathematics, University of the Basque Country UPV/EHU, Barrio Sarriena s/n, 48940, Leioa, Spain. ORCID: 0000-0002-6583-866X E-mail: {\tt jonasier.barcena@ehu.eus}.}\ \thanks{Basque Center for Applied Mathematics (BCAM), Alameda Mazarredo 14, 48009, Bilbao, Bizkaia, Spain} \and  Marius Tucsnak\thanks{
Institut de Mathématiques de Bordeaux,  UMR 5251, Universit\'e de Bordeaux/Bordeaux INP/CNRS,
351, Cours de la Libération - F 33 405 TALENCE, France, and Institut Universitaire de France (IUF). 
ORCID: 0000-0001-9410-1811.
E-mail: {\tt marius.tucsnak@u-bordeaux.fr}.}
}

\maketitle

\begin{abstract}
We study homogenization of norm-optimal and time-optimal controls for a one-dimensional heat equation with rapidly oscillating periodic diffusion coefficients. The controls act on a fixed subinterval and are subject to pointwise constraints in space and time.

We first derive final-state observability and null-controllability estimates that are uniform with respect to the oscillation scale. Combining these estimates with periodic homogenization and an abstract theory of optimal controls for parabolic systems, we prove convergence of the norm-optimal values and controls to their homogenized counterparts. We then deduce convergence of the optimal times and time-optimal controls. The controls converge weakly-star in the space of essentially bounded functions and strongly in every Lebesgue space with finite exponent, after extension to a common time interval when necessary.
\end{abstract}



{\bf Funding.} This work was done while J.A.B.P. was doing a scientific stay at the University of Bordeaux funded by the CAS24/225 ``José Castillejo'' grant, within the framework of the State Subprogramme for Training, Attraction and Retention of Talent under PEICTI 2024 (Spanish Strategy for Science, Technology and Innovation) of the Ministry of Universities (MUNI) of the Spanish Government.  
J.A.B.P. was also supported by the Grant PID2023-146764NB-I00 funded by MICIU/AEI/10.13039/501100011033 and cofunded by the European Union and by the grant~IT1875-26 funded by the Basque Government.

%\tableofcontents
%%%%%%%%%%%%%%%%%%%%%%%%%%%%%%%%%%%%%%%%%%%%%%%%%%%%%%%%%%%%%%%%%%%
\section{Introduction}\label{sec_intro}
\setcounter{equation}{0}

Let $a\in L^\infty(\mathbb R)$ be a $1$-periodic function such that, for some constants $K>0$, we have
\begin{equation}\label{low_coef_c}
K^{-1}  \leqslant a(x)\leqslant K  \qquad(x\in \mathbb R\ {\rm a.e.}).
\end{equation}
Let $I\subset (0,1)$ be an open interval, let $\chi_I$ be its characteristic function, and fix
$z_0\in L^2(0,1)\setminus\{0\}$.
For $\varepsilon>0$, we consider the controlled equation
\begin{equation}\label{heat_osc_aux}
\frac{\partial z}{\partial t}(t,x) -\frac{\partial}{\partial x}\left(a\left(\frac{x}{\varepsilon}\right)\frac{\partial z}{\partial x}(t,x)
\right)=u(t,x) \chi_{I}(x) \qquad(t\geqslant 0,\ \ x\in (0,1)),
\end{equation}
\begin{equation}\label{bound_osc_aux}
z(t,0)=z(t,1)=0 \qquad\qquad(t\geqslant 0),
\end{equation}
\begin{equation}\label{init_osc_aux}
z(0,x)=z_0(x) \qquad\qquad(x\in (0,1)),
\end{equation}
and denote its solution by $z_\varepsilon(\cdot;z_0,u)$.

The index $\varepsilon=0$ will refer to the homogenized equation, in which
$a(x/\varepsilon)$ is replaced by the constant
\begin{equation*}\label{eq:homogenized-coefficient-introduction}
a_{0}
=\left(\int_0^1\frac{1}{a(y)}\,{\rm d}y\right)^{-1}.
\end{equation*}

We study two optimal control problems. First, for a prescribed final time
$\tau>0$, the $\varepsilon$-norm-optimal control problem is
\begin{equation}\label{eq:introduction-norm-optimal-problem}
N_\varepsilon(\tau)
=\inf\left\{
\|u\|_{L^\infty((0,\tau)\times I)}:\
z_\varepsilon(\tau;z_0,u)=0
\right\}.
\end{equation}
A control attaining this infimum will be denoted by $u_\varepsilon^\tau$.
Second, for a prescribed amplitude bound $M>0$, the
$\varepsilon$-time-optimal control problem is
\begin{equation}\label{eq:introduction-time-optimal-problem}
\tau_\varepsilon^*(M)
=\inf\left\{
\tau>0:\ \text{there exists }u\in L^\infty((0,\tau)\times I),\ 
\|u\|_{L^\infty((0,\tau)\times I)}\leqslant M,\ 
z_\varepsilon(\tau;z_0,u)=0
\right\}.
\end{equation}
An associated optimal control will be denoted by $u_\varepsilon^*(M)$.
The same notation with $\varepsilon=0$ refers to the homogenized equation.

%\textcolor{blue}{On a first approach, we may consider the operator $\frac{d}{dx}\left(\frac{1}{a}\frac{d}{dx}\right)$ instead, to ensure that we are working with a self-adjoint operator, which simplifies things. Alternatively, we could say that the operator is $a_\varepsilon A$, for $a_\varepsilon$ a continuous operator from $X$ to $X$.} 

Our first main result concerns the convergence of the solutions of the norm optimal problems.

\begin{theorem}\label{thm:introduction-norm-optimal}
Let $\tau>0$ and $z^0 \in L^2(0,1)$. For every $\varepsilon\geqslant0$, the norm-optimal problem
\eqref{eq:introduction-norm-optimal-problem} admits a unique solution
$u_\varepsilon^\tau$, which is bang-bang:
\begin{equation*}\label{eq:introduction-bang-bang-norm}
|u_\varepsilon^\tau(t,x)|=N_\varepsilon(\tau)
\quad\text{for almost every }(t,x)\in(0,\tau)\times I.
\end{equation*}
Moreover,
\begin{equation*}\label{eq:introduction-convergence-norm-values}
N_\varepsilon(\tau)\longrightarrow N_0(\tau),
\end{equation*}
and
\begin{equation*}\label{eq:introduction-convergence-norm-controls}
u_\varepsilon^\tau\xrightharpoonup{*}u_0^\tau
\quad\text{in }L^\infty((0,\tau)\times I),
\qquad
u_\varepsilon^\tau\longrightarrow u_0^\tau
\quad\text{in }L^p((0,\tau)\times I)
\end{equation*}
for every $1\leqslant p<\infty$, as $\varepsilon\to0$.
\end{theorem}

The second main result concerns the convergence of the solutions of time optimal control problems.

\begin{theorem}\label{th_eps_fix_c}
For every $M>0$, every $z^0 \in L^2(0,1)$ and every $\varepsilon\geqslant0$, the time-optimal problem
\eqref{eq:introduction-time-optimal-problem} admits a unique optimal control
$u_\varepsilon^*(M)$. It is bang-bang:
\begin{equation*}\label{mare_bang_bang}
|u_\varepsilon^*(M)(t,x)|=M
\quad\text{for almost every }(t,x)\in
(0,\tau_\varepsilon^*(M))\times I.
\end{equation*}
There exists $\widehat\tau>0$, depending only on $M$, $|I|$, $K$ and
$\|z_0\|_{L^2(0,1)}$, such that
$\tau_\varepsilon^*(M)\leqslant\widehat\tau$ for every $\varepsilon\geqslant0$.
Furthermore,
\begin{equation*}\label{eq:introduction-convergence-optimal-times}
\tau_\varepsilon^*(M)\longrightarrow\tau_0^*(M).
\end{equation*}
After extending every $u_\varepsilon^*(M)$ by zero to
$(0,\widehat\tau)\times I$, these extensions converge weakly-* in
$L^\infty((0,\widehat\tau)\times I)$ and strongly in
$L^p((0,\widehat\tau)\times I)$, for every $1\leqslant p<\infty$, to the
corresponding extension of $u_0^*(M)$ as $\varepsilon\to0$.
\end{theorem}

The interaction between homogenization and controllability has been extensively studied over the last three decades. For hyperbolic equations, the pioneering work of Avellaneda, Bardos and Rauch \cite{AvellanedaBardosRauch} revealed that controllability properties may be strongly affected by rapidly oscillating coefficients. In the one-dimensional wave equation, Castro and Zuazua and, subsequently, Castro analyzed the asymptotic behavior of the spectrum and of the controls in heterogeneous media; see \cite{CastroZuazua1997,Castro1999,CastroZuazua2000}. Since the controls obtained by the Hilbert Uniqueness Method are controls of minimal $L^2$-norm, the convergence results in these works may also be interpreted as homogenization results for norm-optimal controls with the exact target $\{0\}$, in an $L^2$ setting. They also show that exact controllability need not be uniform with respect to the homogenization parameter. More recently, Lin and Shen \cite{LinShen2022} combined quantitative homogenization with spectral analysis to establish uniform boundary controllability for suitable low-frequency components of wave equations with rapidly oscillating periodic coefficients.

For parabolic equations, the situation is markedly different due to the dissipative nature of the dynamics. A fundamental contribution is the paper \cite{Lop_Zuaz} by López and Zuazua, where the one-dimensional heat equation with rapidly oscillating periodic density was considered. Using a three-step strategy combining low-frequency control, free evolution, and a final null-controllability argument based on Carleman estimates, they proved uniform null controllability and the convergence of minimal $L^2$ norm null controls toward those of the homogenized equation. This result highlighted a striking contrast with the hyperbolic case, where the corresponding controls may blow up as the oscillation scale tends to zero.

A major step toward the treatment of heterogeneous parabolic equations with low regularity coefficients was achieved by Alessandrini and Escauriaza  in \cite{alessandrini2008null}. They proved null controllability for a broad class of one-dimensional parabolic equations with merely measurable, uniformly elliptic coefficients. Their approach relies on quantitative unique continuation properties and yields observability estimates that are robust with respect to the regularity of the coefficients. From the perspective of homogenization, this robustness is particularly important since rapidly oscillating coefficients typically generate families of equations that are not uniformly regular in the oscillation parameter. The results of \cite{alessandrini2008null} therefore provide a natural framework for studying control problems in heterogeneous media and have influenced subsequent works on uniform observability, controllability from measurable sets, and asymptotic problems involving oscillatory coefficients.

Parallel to these developments, a substantial literature has emerged on bang-bang properties and optimal control problems for parabolic equations. The origins of the bang-bang principle can be traced back to the seminal work of Pontryagin and his collaborators \cite{Pontryagin}, where it was shown, within the framework of the maximum principle, that time-optimal controls typically saturate the admissible control constraints. Its extension to infinite-dimensional systems governed by partial differential equations is considerably more delicate. We refer to the monograph \cite{WangEtAl2018} for a systematic account of time-optimal control of evolution equations, including existence, maximum principles, relations between time- and norm-optimal problems, and bang-bang properties.

For parabolic equations, early results mainly concerned time-optimal controls in an $L^2$ setting. A particularly influential contribution is the work \cite{wang_linf} of Wang, who established suitable null-controllability properties for the heat equation and deduced a bang-bang principle for time-optimal controls with values in bounded subsets of $L^2(\Omega)$. The equivalence between minimal-time and minimal-norm null-control problems was studied by Wang and Zuazua \cite{WangZuazua}. Micu, Roventa and Tucsnak \cite{MRT} proved the bang-bang property for time-optimal boundary controls of the heat equation, while Kunisch and Wang \cite{KunischWang2013} investigated the heat equation with pointwise control constraints.

Stability and approximation of time-optimal parabolic controls have also received considerable attention. Yu \cite{Yu} studied perturbations of the potential in the heat equation. Tucsnak, Wang and Wu \cite{tucsnak2016perturbations} developed an abstract perturbation theory for parabolic generators and applied it to equations with rapidly oscillating coefficients. Their application is already a homogenization theorem for time-optimal controls: the controls are constrained in $L^\infty(0,\infty;L^2(\Omega))$, and the target is a closed ball of positive radius centered at the origin. More general stability questions for linear parabolic time-optimal problems and convex terminal sets were studied in \cite{BonifaciusPieper2019}. These works are closely related to the present one, but their control constraints and/or target sets differ from ours.

More recently, observability inequalities from measurable sets and quantitative unique continuation estimates have led to a unified approach to optimal control problems for parabolic equations. Within this framework, existence, uniqueness and bang-bang properties follow from suitable observability estimates. The recent work \cite{tucsnak2025norm} develops an abstract theory of norm-optimal and time-optimal controls in space--time $L^\infty$ for a broad class of linear parabolic systems.

The novelty of the present paper should therefore be understood in the following precise sense. We study homogenization when the admissible controls satisfy a pointwise constraint in space and time, that is, they belong to a ball of $L^\infty((0,\tau)\times I)$, and when the target is the single point $\{0\}$ rather than a ball with nonempty interior. In this setting we treat both norm-optimal and time-optimal controls and prove convergence to their homogenized counterparts. To the best of our knowledge, this combination of an exact null target and space--time $L^\infty$ constraints has not previously been considered in homogenization. Our proof combines the abstract framework of \cite{tucsnak2025norm}, uniform observability from measurable sets, and periodic homogenization. The exact target is important: unlike a ball of positive radius, it has empty interior, so the short correction furnished by uniform null controllability becomes an essential part of the convergence argument.

The paper is organized as follows.  Section
\ref{sec:spectral-observability} contains parameter-free spectral and
observability estimates, including observation on arbitrary measurable
sets of positive measure.  Section \ref{sec:abstract-time-optimal}
collects the abstract facts concerning norm- and time-optimal controls.
Section \ref{sec:operator-approximation} proves homogenization of the
norm-optimal values and controls. Section
\ref{sec:time-optimal-homogenization} then deduces the convergence of the
time-optimal values, controls and trajectories.

%%%%%%%%%%%%%%%%%%%%%%%%%%%%%%%%%%%%%%%%%%%%%%%%%%%%%%%%%%%%%%%%%%%%%%%%%%%%%%%%%%%%%%%%%%%%%%%%%%%%%%%%%%%%%%%%%%%%%%%%%%%%%%%%%%%%%%%%%%%%%%

In this paper, given a set $G$ in $\mathbb R^d$ for some $d\geqslant1$, $|G|$ denotes its Lebesgue measure.

\section{Final state observability for the heat equation with rough coefficients}
\label{sec:spectral-observability}

Let $a\in L^\infty(0,1)$ be real-valued and assume
that two fixed constants $a_0,a_1>0$ satisfy
\begin{equation}\label{eq:ellipticity-a}
a_0\leqslant a(x)\leqslant a_1
\qquad\text{for almost every }x\in(0,1).
\end{equation}
Consider the closed symmetric form
\begin{equation*}\label{eq:form-a}
\mathfrak a_a(v,w)=\int_0^1a(x)v'(x)\overline{w'(x)}\,{\rm d}x,
\qquad v,w\in H_0^1(0,1),
\end{equation*}
and denote by $A_a$ its associated strictly positive  operator. More precisely,
\[
\mathcal{D}(A_a)=\left\{\varphi\in H_0^1(0,1)\ \ |\ \ a \frac{{\rm d}\varphi}{{\rm d}x}\in H^1(0,1)\right\}
\]
\begin{equation}\label{eq:variable-coefficient-operator}
A_a v=-\frac{{\rm d}}{{\rm d}x}\left(a\frac{{\rm d}v}{{\rm d}x}\right) \qquad\qquad(v\in \mathcal{D}(A_a)).
\end{equation}
Since the resolvents of $A_a$ are obviously compact it follows that
there exists an orthonormal basis $(e_k^a)_{k\geqslant 1}$ of $L^2(0,1)$ 
formed of eigenvectors of $A_a$. We denote by $(\lambda_k^a)_{k\in \mathbb{N}}$
the corresponding sequence of eigenvalues.

\begin{proposition}\label{prop:L1-spectral-variable-coefficient}
For every $m\in(0,1]$, there exists $N=N(a_0,a_1,m)\geqslant 1$ such that,
for every $a$ satisfying \eqref{eq:ellipticity-a}, every measurable
$\mathcal O\subset(0,1)$ with $|\mathcal O|\geqslant m$, every $\mu\geqslant 1$,
and every $(b_k)\in l^2(\mathbb{C})$,
\begin{equation}\label{eq:L1-spectral-variable-coefficient}
\left(\sum_{\sqrt{\lambda_k^a}\leqslant\mu}|b_k|^2\right)^{1/2}
\leqslant N\exp(N\mu)
\left\|\sum_{\sqrt{\lambda_k^a}\leqslant\mu}b_k e_k^a
\right\|_{L^1(\mathcal O)}.
\end{equation}
\end{proposition}

\begin{proof}
We first deduce from \cite[Proposition~A.1]{zhu2025propagation} the
following $L^\infty$ spectral inequality. For every
$m_0>0$ there exists $C=C(a_0,a_1,m_0)\geqslant1$ such that
\begin{equation}\label{eq:Zhu-Linfty-spectral}
 \|g\|_{L^\infty(0,1)}
 \leqslant \exp(C\mu)\|g\|_{L^\infty(E)}
\end{equation}
for every measurable $E\subset(0,1)$ with $|E|\geqslant m_0$ and every
\[
 g\in\operatorname{span}\{e_k^a:\sqrt{\lambda_k^a}\leqslant\mu\}.
\]
The constant in \eqref{eq:Zhu-Linfty-spectral} is uniform with respect
to $a$ satisfying \eqref{eq:ellipticity-a}.

Inequality \eqref{eq:Zhu-Linfty-spectral} follows directly from
\cite[Proposition~A.1]{zhu2025propagation} by the standard reflection and
rescaling argument. More precisely, one extends $a$ evenly and the
Dirichlet eigenfunctions $e_k^a$ oddly across $0$ and $1$, and then extends
both functions $2$-periodically. The extended eigenfunctions satisfy the
same eigenvalue equation for the periodic operator
$-\partial_x(\widetilde a\,\partial_x)$ on
$\mathbb R/(2\mathbb Z)$. After rescaling this torus to the unit torus,
the eigenvalues are multiplied by $4$. We can thus apply \cite[Proposition~A.1]{zhu2025propagation},
where the constant depends only on the ellipticity parameter
$\max\{a_1,a_0^{-1}\}$ and on a positive lower bound for
the Lebesgue measure of $E$. This proves \eqref{eq:Zhu-Linfty-spectral} with
$C=C(a_0,a_1,m_0)$. 

We now pass from \eqref{eq:Zhu-Linfty-spectral} to the required
$L^1$ estimate.  Set
\[
f=\sum_{\sqrt{\lambda_k^a}\leqslant\mu}b_k e_k^a,
\qquad
S=\left(\sum_{\sqrt{\lambda_k^a}\leqslant\mu}|b_k|^2\right)^{1/2},
\]
 and define
\[
 E=\left\{x\in\mathcal O:\ |f(x)|\leqslant
 \frac{2}{|\mathcal O|}\|f\|_{L^1(\mathcal O)}\right\}.
\]
Its complement in $\mathcal O$ is
\[
 \mathcal O\setminus E
 =\left\{x\in\mathcal O:\ |f(x)|>
 \frac{2}{|\mathcal O|}\|f\|_{L^1(\mathcal O)}\right\}.
\]
The measure-theoretic form of Markov's inequality, also often called
Chebyshev's inequality, applied to the nonnegative function
$|f|\chi_{\mathcal O}$ gives (see \cite[Proposition~4.1]{Axler2020})
\[
 |\mathcal O\setminus E|
 \leqslant
 \frac{|\mathcal O|}{2\|f\|_{L^1(\mathcal O)}}
 \int_{\mathcal O}|f(x)|\,{\rm d}x
 =\frac{|\mathcal O|}{2},
\]
with the conclusion being immediate if
$\|f\|_{L^1(\mathcal O)}=0$. Consequently,
\[
 |E|\geqslant\frac{|\mathcal O|}{2}\geqslant\frac m2.
\]
Moreover, by the definition of $E$,
\[
 \|f\|_{L^\infty(E)}
 \leqslant
 \frac{2}{|\mathcal O|}\|f\|_{L^1(\mathcal O)}
 \leqslant
 \frac{2}{m}\|f\|_{L^1(\mathcal O)}.
\]
Applying \eqref{eq:Zhu-Linfty-spectral} with $m_0=m/2$ and combining it
with the preceding estimate therefore yields
\begin{equation}\label{eq:Zhu-L1-spectral}
 \|f\|_{L^\infty(0,1)}
 \leqslant \frac{2}{m}\exp(C\mu)\|f\|_{L^1(\mathcal O)},
 \qquad C=C(a_0,a_1,m).
\end{equation}
Finally, orthonormality of $(e_k^a)$ gives
\[
 S=\|f\|_{L^2(0,1)}\leqslant\|f\|_{L^\infty(0,1)}.
\]
Combining this with \eqref{eq:Zhu-L1-spectral} and increasing the constant
proves \eqref{eq:L1-spectral-variable-coefficient}.
\end{proof}

\begin{remark}\label{rem:comparison-Apraiz-Escauriaza-Zhu}
The argument above is closely related to the rough-coefficient
one-dimensional construction of Apraiz and Escauriaza
\cite{apraiz2013null}. After a bilipschitz change of variables and periodic
reflection, they introduce a two-dimensional elliptic extension and prove a
tensorized spectral inequality in which the extension is observed on
$\mathcal O\times(1/4,3/4)$. Combined with a Lebeau--Robbiano construction,
this yields null controllability from measurable sets. That tensorized
inequality does not directly give an estimate of the trace at $y=0$ in terms
of its values on $\mathcal O$, because the observation still involves the
full elliptic extension for $y\in(1/4,3/4)$. By contrast,
\cite[Proposition~A.1]{zhu2025propagation}, itself obtained from Zhu's
two-dimensional propagation of smallness from subsets of a line, gives
directly the periodic trace estimate used in
\eqref{eq:Zhu-Linfty-spectral}. The only additional steps required here are
the passage from Dirichlet to periodic boundary conditions by reflection and
the elementary level-set argument converting the $L^\infty$ estimate into
the $L^1(\mathcal O)$ inequality
\eqref{eq:L1-spectral-variable-coefficient}. This last norm is the natural
one for the duality with $L^\infty$ controls.
\end{remark}

We finish with the observability and controllability consequences.

\begin{corollary}\label{cor:variable-coefficient-observability}
Let $\tau>0$ and let $\mathcal D\subset(0,\tau)\times(0,1)$ be measurable, with
$|\mathcal D|>0$. There exists $K=K(a_0,a_1,\tau,\mathcal D)>0$, independent
of the particular coefficient $a$, such that
\begin{equation*}\label{eq:variable-coefficient-final-observability}
\|\exp(-\tau A_a)f\|_{L^2(0,1)}
\leqslant K\int_{\mathcal D}|\exp(-tA_a)f(x)|\,{\rm d}x\,{\rm d}t.
\end{equation*}
\end{corollary}

\begin{proof}
Apply Theorem \ref{theo:poate_abs} from the Appendix and Proposition
\ref{prop:L1-spectral-variable-coefficient}.
\end{proof}

In order to formulate the dual controllability result, we now introduce the
reflection of the control set with respect to the midpoint of the time
interval:
\begin{equation*}\label{eq:time-reflected-set}
\mathcal D^\sharp=\{(\tau-t,x):(t,x)\in\mathcal D\}.
\end{equation*}
The map $(t,x)\mapsto(\tau-t,x)$ preserves Lebesgue measure. Consequently,
\begin{equation}\label{eq:measure-reflected-set}
|\mathcal D^\sharp|=|\mathcal D|.
\end{equation}
Moreover, if $\mathcal D_t=\{x\in(0,1):(t,x)\in\mathcal D\}$ denotes the
spatial section of $\mathcal D$ at time $t$, then
\begin{equation*}\label{eq:sections-reflected-set}
(\mathcal D^\sharp)_t=\mathcal D_{\tau-t}
\qquad\text{for almost every }t\in(0,\tau).
\end{equation*}

\begin{corollary}\label{cor:variable-coefficient-null-control}
Under the assumptions of Corollary
\ref{cor:variable-coefficient-observability}, for every
$z^0\in L^2(0,1)$ there exists
$u\in L^\infty(\mathcal D)$ such that the mild solution of
\begin{equation*}\label{eq:controlled-variable-coefficient-heat}
\left\{
\begin{aligned}
\frac{\partial z}{\partial t}(t,x)+(A_a z)(t,x)&=\chi_{\mathcal D}(t,x) u(t,x)
&& \qquad((t,x)\in(0,\tau)\times(0,1)),\\
z(t,0)=z(t,1)&=0&& \qquad(t\in(0,\tau)),\\
z(0,\cdot)&=z^0
\end{aligned}
\right.
\end{equation*}
satisfies $z(\tau,\cdot)=0$. Moreover,
\begin{equation}\label{eq:Linfty-control-cost}
\|u\|_{L^\infty(\mathcal D)}
\leqslant K(a_0,a_1,\tau,\mathcal D^\sharp)\|z^0\|_{L^2(0,1)}.
\end{equation}
For $\mathcal D=(0,\tau)\times\mathcal O$, $|\mathcal O|\geqslant m$, one may take
\begin{equation}\label{eq:cylindrical-Linfty-control-cost}
\|u\|_{L^\infty((0,\tau)\times\mathcal O)}
\leqslant C\exp\left(\frac{C}{\tau}\right)\|z^0\|_{L^2(0,1)},
\qquad C=C(a_0,a_1,m).
\end{equation}
\end{corollary}

\begin{proof}
The adjoint of the input-to-state map at time $\tau$ is
\[
f\longmapsto\chi_{\mathcal D}(t,x)
\bigl(\exp(-(\tau-t)A_a)f\bigr)(x).
\]
Apply Corollary \ref{cor:variable-coefficient-observability} to
$\mathcal D^\sharp$ and use Hahn--Banach duality  between
$L^1(\mathcal D)$ and $L^\infty(\mathcal D)$ (see, for instance, \cite[Proposition 2.2]{micu2012time}). Conversely, the control
estimate implies the adjoint observability inequality. Thus the optimal
final-state observability constant for $\mathcal D^\sharp$ equals the
optimal $L^\infty$ null-control cost for $\mathcal D$.

Notice that when $\mathcal D=(0,\tau)\times\mathcal O$, one has
$\mathcal D^\sharp=\mathcal D$. Hence no time reflection appears in the
cylindrical estimate \eqref{eq:cylindrical-Linfty-control-cost}.
\end{proof}

\begin{remark}\label{rem:cost-dependence-variable-coefficient}
For a general measurable space--time set $\mathcal D$, the constant produced
by the proof depends on the distribution in time of the sections
$\mathcal D_t$, and not merely on $|\mathcal D|$. Once $a_0$, $a_1$, $\tau$,
and $\mathcal D$ are fixed, however, both the observability constant and the
control cost are uniform with respect to $a$. Although
\eqref{eq:measure-reflected-set} shows that $\mathcal D$ and
$\mathcal D^\sharp$ have the same measure, the telescoping construction for
a general space--time set also uses the distribution in time of its spatial
sections. Time reflection reverses this distribution, and the proof does not
by itself identify the particular constants produced for $\mathcal D$ and
$\mathcal D^\sharp$. This is why the constant associated with
$\mathcal D^\sharp$ is retained in \eqref{eq:Linfty-control-cost}.
\end{remark}

%%%%%%%%%%%%%%%%%%%%%%%%%%%%%%%%%%%%%%%%%%%%%%%%%%%%%%%%%%%%%%%%%%%%%%%%%%%%%%%%%%%%%%%%%%%%%%%%%%%%%%%%%%%%%%%%%%%%%%%%%%%%%%%%%%%%%%%%%%%

\section{Abstract results on norm- and time-optimal controls}
\label{sec:abstract-time-optimal}

In this section we recall the abstract framework and the consequences of
\cite{tucsnak2025norm} that will be used below. Let $X$ and $U$ be Hilbert
spaces, let $A:\mathcal D(A)\to X$ be such that $-A$ generates a strongly continuous semigroup
$(\mathbb T_t)_{t\geqslant0}$ on $X$, and let $B\in\mathcal L(U,X)$. We
consider the control system
\begin{equation*}\label{eq:abstract-differential-system}
 \dot z(t)+Az(t) = Bu(t),
 \qquad z(0)=z^0.
\end{equation*}
Its mild solution is
\begin{equation*}\label{eq:abstract-controlled-system}
 z(t)=\mathbb T_tz^0+\Phi_tu,
\end{equation*}
where the family $(\Phi_t)_{t>0}$ of input maps is defined by
\begin{equation*}\label{eq:abstract-input-maps}
 \Phi_t\in\mathcal L(L^2(0,t;U),X),
 \qquad
 \Phi_t u=\int_0^t\mathbb T_{t-s}Bu(\sigma)\,{\rm d}\sigma.
\end{equation*}
Since $B$ is bounded, these maps are well defined. Their adjoints are given
by
\begin{equation*}\label{eq:abstract-adjoint-input-maps}
 (\Phi_t^*f)(s)=B^*\mathbb T_{t-s}^*f
 \qquad(0<s<t,\ f\in X).
\end{equation*}
When controls are regarded as functions on $(0,\infty)$, both $u$ and
$\Phi_t^*f$ are extended by zero outside $(0,t)$.

We shall use the following assumptions.
\begin{enumerate}
\item[{\rm(H1)}] $U=L^2(\mathcal O)$, where $\mathcal O$ is a measurable
set of positive measure in $\mathbb R^d$ for some $d\in\mathbb N$.
\item[{\rm(H2)}] The semigroup $(\mathbb T_t)_{t\geqslant0}$ is analytic
and its generator $-A$ has compact resolvent.
\item[{\rm(H3)}] For every $\tau>0$ and every measurable set
$\mathcal D\subset(0,\tau)\times\mathcal O$ of positive measure, there is
$K_{\tau,\mathcal D}>0$ such that
\begin{equation*}\label{eq:abstract-H3}
 \|\mathbb T_\tau^*f\|_X
 \leqslant K_{\tau,\mathcal D}
 \int_{\mathcal D}|(B^*\mathbb T_{\tau-t}^*f)(x)|\,{\rm d}x\,{\rm d}t
 \qquad(f\in X).
\end{equation*}
\end{enumerate}

Fix $z^0\in X\setminus\{0\}$. For $\tau>0$, define the norm-optimal
value
\begin{equation}\label{eq:abstract-norm-optimal-value}
 N(\tau)=\inf\left\{
 \|u\|_{L^\infty((0,\tau)\times\mathcal O)}:
 \mathbb T_\tau z^0+\Phi_\tau u=0
 \right\}.
\end{equation}
The following result collects the part of
\cite[Theorems 3.2 and 3.3]{tucsnak2025norm} needed in this paper.

\begin{theorem}\label{tm:normoptcon}
Assume {\rm(H1)--(H3)}. For every $\tau>0$, the infimum in
\eqref{eq:abstract-norm-optimal-value} is attained by a unique control
$u^\tau$. Moreover, this control is bang-bang:
\begin{equation*}\label{eq:abstract-bang-bang-norm}
 |u^\tau(t,x)|=N(\tau)
 \quad\text{for almost every }(t,x)\in(0,\tau)\times\mathcal O.
\end{equation*}
The function $N:(0,\infty)\to(0,\infty)$ is continuous and strictly
decreasing. If $(\mathbb T_t)_{t\geqslant0}$ is exponentially stable, then
\begin{equation*}\label{eq:abstract-N-decay}
 \lim_{\tau\to\infty}N(\tau)=0.
\end{equation*}
\end{theorem}

For $M>0$, define
\begin{equation*}\label{eq:abstract-time-optimal-value}
 \tau^*(M)=\inf\left\{
 \tau>0:\ \text{there exists }u\in L^\infty((0,\tau)\times\mathcal O),
 \ \|u\|_{L^\infty}\leqslant M,
 \ \mathbb T_\tau z^0+\Phi_\tau u=0
 \right\}.
\end{equation*}
The corresponding time-optimal result is
\cite[Theorem 3.4]{tucsnak2025norm}.

\begin{theorem}\label{tm:timeoptcon}
Assume {\rm(H1)--(H3)} and
$M>\lim_{\tau\to\infty}N(\tau)$. Then
\begin{equation}\label{eq:abstract-time-norm-relation}
 N(\tau^*(M))=M.
\end{equation}
The time-optimal problem admits a unique optimal control $u^*$, and
\begin{equation*}\label{eq:abstract-bang-bang-time}
 |u^*(t,x)|=M
 \quad\text{for almost every }(t,x)\in(0,\tau^*(M))\times\mathcal O.
\end{equation*}
\end{theorem}

The two last results will play a central role in the proof of our main theorems. More precisely,
the convergence of norm-optimal controls for oscillating coefficients is based on the existence,
uniqueness and the bang-bang property from Theorem \ref{tm:normoptcon}, as
well as the continuity of $N$. The convergence of time-optimal controls
uses, in addition, the fact that $N$ is strictly decreasing, the identity
\eqref{eq:abstract-time-norm-relation}, together with uniqueness and the bang-bang
property from Theorem \ref{tm:timeoptcon}. 

\section{Homogenization of norm-optimal controls}
\label{sec:operator-approximation}

In this section we consider only the periodic homogenization problem arising
from \eqref{heat_osc_aux}. Our purpose is to prove the convergence of the
norm-optimal controls in a fixed time. The convergence of time-optimal
controls will be considered in the next section.

Uniform null-controllability and convergence of $L^2$ norm minimal controls
for heat equations with rapidly oscillating
coefficients was established for controls in $L^2$
by L\'opez and Zuazua in
\cite{Lop_Zuaz}, by means of a three-step construction. That construction is
not needed here: the uniform spectral and observability inequalities obtained
in Section \ref{sec:spectral-observability} give the required uniform control
estimate directly. We combine this estimate with the norm-optimal control
results recalled in Section \ref{sec:abstract-time-optimal} and with standard
one-dimensional periodic homogenization.

Let $I\subset (0;1)$ be a non empty open interval, let $X=L^2(0,1)$, $U=L^2(I)$, and let
\begin{equation}\label{eq:control-operator-homogenization}
B\in\mathcal L(U,X),\qquad (Bv)(x)=\chi_I(x)v(x).
\end{equation}
For $\varepsilon>0$, set
\begin{equation*}\label{eq:oscillating-coefficient}
a_\varepsilon(x)=a\left(\frac{x}{\varepsilon}\right),
\end{equation*}
where $a$ is extended to $\mathbb R$ as a $1$-periodic function, and
define
\begin{equation*}\label{eq:homogenized-coefficient}
a_{0}
=\left(\int_0^1\frac{1}{a(y)}\,{\rm d}y\right)^{-1}.
\end{equation*}

For every
$\varepsilon\geqslant0$, let $A_\varepsilon$ be the positive self-adjoint
operator on $X$ associated with the closed form
\begin{equation*}\label{eq:forms-A-epsilon}
\mathfrak a_\varepsilon(v,w)
=\int_0^1 a_\varepsilon(x) \frac{{\rm d}v}{{\rm d}x} (x)\overline{\frac{{\rm d}w}{{\rm d}x} (x)}\,{\rm d}x
\qquad\qquad (v,w\in H_0^1(0,1)).
\end{equation*}
Thus, in the sense of distributions,
\begin{equation*}\label{eq:operators-A-epsilon}
A_\varepsilon v
=-\frac{{\rm d}}{{\rm d}x}
\left(a_\varepsilon\frac{{\rm d}v}{{\rm d}x}\right),
\qquad
A_0v=-a_{0}\frac{{\rm d}^2v}{{\rm d}x^2},
\end{equation*}
with homogeneous Dirichlet boundary conditions. We write
\begin{equation*}\label{eq:semigroups-epsilon}
\mathbb T_t^\varepsilon=\exp(-tA_\varepsilon)
\qquad\qquad(t\geqslant0,\ \varepsilon\geqslant0)
\end{equation*}
and
\begin{equation*}\label{eq:input-state-epsilon}
\Phi_\tau^\varepsilon u
=\int_0^\tau\mathbb T_{\tau-t}^\varepsilon Bu(t)\,{\rm d}t \qquad\qquad(u\in L^2([0,\infty);U)).
\end{equation*}
Thus, in the notation of Section \ref{sec:abstract-time-optimal}, the
generator of the $\varepsilon$-system is $-A_\varepsilon$; the symbol
$A_\varepsilon$ itself denotes the positive elliptic operator.

\subsection{Uniform verification of the abstract assumptions}

We first verify that the hypotheses recalled in Section
\ref{sec:abstract-time-optimal} hold uniformly for the family
$(A_\varepsilon)_{\varepsilon\geqslant0}$.

\begin{proposition}\label{prop:uniform-H1-H3}
The pairs $(A_\varepsilon,B)$ satisfy assumptions
{\rm(H1)--(H3)} of Section \ref{sec:abstract-time-optimal}, uniformly with respect to $\varepsilon$. More precisely:
\begin{enumerate}
\item The state and input spaces $X=L^2(0,1)$ and $U=L^2(I)$, as well as
the control operator $B$, are independent of $\varepsilon$.
\item For all $\varepsilon\geqslant 0$ the operator $-A_\varepsilon$ is positive, self-adjoint and with compact resolvent,
thus $(\mathbb T_t^\varepsilon)_{t\geqslant0}$ is an analytic contraction
semigroup.
\item For every $\tau>0$ and every measurable
$\mathcal D\subset(0,\tau)\times I$ with $|\mathcal D|>0$, there exists
$K_{\tau,\mathcal D}>0$, independent of $\varepsilon\geqslant0$, such that
\begin{equation}\label{eq:uniform-H3}
\|\mathbb T_\tau^\varepsilon f\|_{L^2(0,1)}
\leqslant K_{\tau,\mathcal D}
\int_{\mathcal D}
|\mathbb T_{\tau-t}^\varepsilon f(x)|\,{\rm d}x\,{\rm d}t
\qquad(f\in L^2(0,1)).
\end{equation}
\end{enumerate}
\end{proposition}

\begin{proof}
The first assertion is immediate from
\eqref{eq:control-operator-homogenization}. The bounds
\begin{equation}\label{eq:uniform-form-bounds}
K^{-1}\|v'\|_{L^2(0,1)}^2
\leqslant \mathfrak a_\varepsilon(v,v)
\leqslant K\|v'\|_{L^2(0,1)}^2
\qquad(v\in H_0^1(0,1))
\end{equation}
hold uniformly in $\varepsilon$ from \eqref{low_coef_c}. The representation theorem for closed
forms therefore shows that $A_\varepsilon$ is positive and self-adjoint.
Its resolvent is compact because
$H_0^1(0,1)\hookrightarrow L^2(0,1)$ is compact. Positivity and
self-adjointness imply analyticity and contractivity of the semigroup. In
fact, the spectral theorem also gives the uniform estimate
\begin{equation*}\label{eq:uniform-analyticity}
\|A_\varepsilon\mathbb T_t^\varepsilon\|_{\mathcal L(X)}
\leqslant \frac{1}{et}\qquad(t>0).
\end{equation*}

Finally,
\begin{equation*}\label{eq:adjoint-input-map-epsilon}
((\Phi_\tau^\varepsilon)^*f)(t,x)
=\chi_I(x)(\mathbb T_{\tau-t}^\varepsilon f)(x).
\end{equation*}
Apply Corollary \ref{cor:variable-coefficient-observability} to the
time-reflected set $\mathcal D^\sharp$. Since all coefficients
$a_\varepsilon$ satisfy the same lower and upper bounds, the resulting
constant is independent of $\varepsilon$. This proves
\eqref{eq:uniform-H3}. 
\end{proof}

\subsection{Homogenization of the semigroups and input maps}

The convergence result needed below follows directly from the operator
error estimates of Meshkova and Suslina. Their result applies to the
present problem by taking $d=m=n=1$, the domain $(0,1)$, $b(D)=D$, and
$g=a$. In this case, the effective coefficient appearing in their work is
the harmonic mean $a_0$ defined above.

\begin{proposition}\label{prop:semigroup-homogenization}
There exist constants $C,c>0$ and $\varepsilon_0>0$, depending only on
the lower and upper bounds for $a$, such that
\begin{equation}\label{eq:operator-norm-semigroup-homogenization}
 \|\mathbb T_t^\varepsilon-\mathbb T_t^0\|_{\mathcal L(L^2(0,1))}
 \leqslant
 C\varepsilon(t+\varepsilon^2)^{-1/2}e^{-ct}
 \qquad(t\geqslant0,\ 0<\varepsilon\leqslant\varepsilon_0).
\end{equation}
Consequently, for every $\tau>0$ and every $f\in L^2(0,1)$,
\begin{equation}\label{eq:strong-semigroup-homogenization}
 \mathbb T_\cdot^\varepsilon f\longrightarrow
 \mathbb T_\cdot^0f
 \quad\text{in }C([0,\tau];L^2(0,1)).
\end{equation}
\end{proposition}

\begin{proof}
Estimate \eqref{eq:operator-norm-semigroup-homogenization} is the scalar
one-dimensional Dirichlet case of
\cite[Theorem~4.2]{MeshkovaSuslina2016}. Notice that this estimate gives
operator-norm convergence uniformly on every compact subinterval of
$(0,\infty)$, but not uniformly down to $t=0$.

To prove \eqref{eq:strong-semigroup-homogenization}, first take
$f\in H_0^1(0,1)$. The spectral theorem and the inequality
$(1-e^{-s})^2\leqslant s$ for $s\geqslant0$ give
\begin{equation*}
 \|\mathbb T_t^\varepsilon f-f\|_{L^2(0,1)}^2
 \leqslant
 t\|A_\varepsilon^{1/2}f\|_{L^2(0,1)}^2
 =t\mathfrak a_\varepsilon(f,f)
 \leqslant Kt\|f'\|_{L^2(0,1)}^2.
\end{equation*}
The same estimate holds for $\mathbb T_t^0$. Hence the convergence is
uniform near $t=0$ for $f\in H_0^1(0,1)$, while
\eqref{eq:operator-norm-semigroup-homogenization} gives uniform
convergence away from zero. The conclusion for arbitrary
$f\in L^2(0,1)$ follows by density of $H_0^1(0,1)$ and contractivity of
the semigroups.
\end{proof}

\begin{lemma}\label{lem:input-map-homogenization}
Let $\tau>0$.
\begin{enumerate}
\item For every $u\in L^2((0,\tau)\times I)$,
\begin{equation}\label{eq:strong-input-map-homogenization}
 \Phi_\tau^\varepsilon u\longrightarrow\Phi_\tau^0u
 \quad\text{strongly in }L^2(0,1).
\end{equation}
\item If
$u_\varepsilon\xrightharpoonup{*}u$ in
$L^\infty((0,\tau)\times I)$, then
\begin{equation}\label{eq:weak-input-map-homogenization}
 \Phi_\tau^\varepsilon u_\varepsilon
 \rightharpoonup\Phi_\tau^0u
 \quad\text{weakly in }L^2(0,1).
\end{equation}
\end{enumerate}
\end{lemma}

\begin{proof}
By \eqref{eq:operator-norm-semigroup-homogenization} and the
Cauchy--Schwarz inequality,
\begin{align*}
 \|\Phi_\tau^\varepsilon u-\Phi_\tau^0u\|_{L^2(0,1)}
 &\leqslant
 C\left(\int_0^\tau
 \frac{\varepsilon^2}{s+\varepsilon^2}\,{\rm d}s\right)^{1/2}
 \|Bu\|_{L^2(0,\tau;L^2(0,1))}\\
 &=C\varepsilon
 \left(\log\left(1+\frac{\tau}{\varepsilon^2}\right)\right)^{1/2}
 \|Bu\|_{L^2(0,\tau;L^2(0,1))}.
\end{align*}
This proves the first assertion and, in fact, the operator-norm
convergence of $\Phi_\tau^\varepsilon$ to $\Phi_\tau^0$ as maps from
$L^2(0,\tau;L^2(I))$ to $L^2(0,1)$.

For the second assertion, write
\begin{equation*}
 \Phi_\tau^\varepsilon u_\varepsilon-\Phi_\tau^0u
 = (\Phi_\tau^\varepsilon-\Phi_\tau^0)u_\varepsilon
   +\Phi_\tau^0(u_\varepsilon-u).
\end{equation*}
The first term converges strongly to zero because
$(u_\varepsilon)$ is bounded in $L^\infty$, hence in $L^2$, and because
of the operator-norm estimate above. Weak-* convergence in $L^\infty$ on
the finite-measure set $(0,\tau)\times I$ implies weak convergence in
$L^2$. Therefore, the second term converges weakly to zero in $L^2(0,1)$.
\end{proof}

\subsection{Proof of Theorem~\ref{thm:introduction-norm-optimal}}

Fix $\tau>0$ and $z^0\in L^2(0,1)\setminus\{0\}$. Define
\begin{equation*}\label{eq:norm-optimal-value-epsilon}
N_\varepsilon(\tau)
=\inf\left\{
\|u\|_{L^\infty((0,\tau)\times I)}:
\mathbb T_\tau^\varepsilon z^0+\Phi_\tau^\varepsilon u=0
\right\}.
\end{equation*}
The function $N_\varepsilon$ also depends on $z^0$, but we suppress this dependence to simplify the notation. Proposition \ref{prop:uniform-H1-H3} and the results recalled in Section
\ref{sec:abstract-time-optimal} show that the infimum is attained by a
unique control $u_\varepsilon^\tau$. Moreover,
\begin{equation}\label{eq:bang-bang-norm-optimal-epsilon}
|u_\varepsilon^\tau(t,x)|=N_\varepsilon(\tau)
\quad\text{for almost every }(t,x)\in(0,\tau)\times I.
\end{equation}

\begin{proposition}\label{prop:uniform-bound-N-epsilon}
For every $\tau>0$ there exists $C_\tau>0$, independent of
$\varepsilon\geqslant 0$, such that
\begin{equation*}\label{eq:uniform-bound-N-epsilon}
N_\varepsilon(\tau)\leqslant C_\tau\|z^0\|_{L^2(0,1)}.
\end{equation*}
\end{proposition}

\begin{proof}
Apply Corollary \ref{cor:variable-coefficient-null-control} to the
coefficient $a_\varepsilon$, with final time $\tau$, control region $I$,
and initial state $z^0$. Since
$K^{-1}\leqslant a_\varepsilon\leqslant K$ uniformly in $\varepsilon$,
the cylindrical estimate \eqref{eq:cylindrical-Linfty-control-cost}
provides a control $u$ satisfying
\[
 \mathbb T_\tau^\varepsilon z^0+\Phi_\tau^\varepsilon u=0,
 \qquad
 \|u\|_{L^\infty((0,\tau)\times I)}
 \leqslant C\exp(C/\tau)\|z^0\|_{L^2(0,1)},
\]
where $C$ depends only on $K$ and $|I|$. Taking the infimum over all
admissible controls gives the asserted estimate, with
$C_\tau=C\exp(C/\tau)$.
\end{proof}

We shall use the following consequence of the same uniform
null-controllability estimate.

\begin{lemma}[Uniform correction property]\label{lem:uniform-correction}
For every $\delta>0$ there exists $C_\delta>0$, independent of
$\varepsilon\geqslant0$, with the following property. For every
$r\in L^2(0,1)$ there exists
$w_{\varepsilon,r}\in L^\infty((0,\delta)\times I)$ such that
\begin{equation}\label{eq:uniform-correction-terminal-equation}
 \mathbb T_\delta^\varepsilon r
 +\Phi_\delta^\varepsilon w_{\varepsilon,r}=0
\end{equation}
and
\begin{equation}\label{eq:uniform-correction-estimate}
 \|w_{\varepsilon,r}\|_{L^\infty((0,\delta)\times I)}
 \leqslant C_\delta\|r\|_{L^2(0,1)}.
\end{equation}
\end{lemma}

\begin{proof}
Apply Corollary \ref{cor:variable-coefficient-null-control} on the time
interval $(0,\delta)$, with the coefficient $a_\varepsilon$ and the
cylindrical control set $(0,\delta)\times I$. Since
$K^{-1}\leqslant a_\varepsilon\leqslant K$ for every $\varepsilon$, the
constant in \eqref{eq:cylindrical-Linfty-control-cost} is independent of
$\varepsilon$. The resulting control satisfies
\eqref{eq:uniform-correction-terminal-equation} and
\eqref{eq:uniform-correction-estimate}; one may take
$C_\delta=C\exp(C/\delta)$, where $C$ depends only on $K$ and $|I|$.
\end{proof}

\begin{proposition}\label{prop:convergence-norm-optimal-values}
For every $\tau>0$,
\begin{equation}\label{eq:convergence-N-epsilon}
\lim_{\varepsilon\to0}N_\varepsilon(\tau)=N_0(\tau).
\end{equation}
\end{proposition}

\begin{proof}
\textit{Step 1. Proof of the lower bound.}
 Let $(\varepsilon_n)$ be a sequence tending
to zero such that
\begin{equation*}\label{eq:choice-liminf-N-epsilon}
N_{\varepsilon_n}(\tau)
\longrightarrow\liminf_{\varepsilon\to0}N_\varepsilon(\tau).
\end{equation*}
By Proposition \ref{prop:uniform-bound-N-epsilon}, the controls
$u_{\varepsilon_n}^\tau$ are bounded in
$L^\infty((0,\tau)\times I)$. The Banach--Alaoglu theorem gives, after
extraction,
\begin{equation*}\label{eq:weak-star-optimal-controls}
u_{\varepsilon_n}^\tau\xrightharpoonup{*}u
\quad\text{in }L^\infty((0,\tau)\times I).
\end{equation*}
Since
\begin{equation*}\label{eq:terminal-equation-epsilon}
\mathbb T_\tau^{\varepsilon_n}z^0
+\Phi_\tau^{\varepsilon_n}u_{\varepsilon_n}^\tau=0,
\end{equation*}
Proposition \ref{prop:semigroup-homogenization} and Lemma
\ref{lem:input-map-homogenization} imply
\begin{equation*}\label{eq:terminal-equation-limit}
\mathbb T_\tau^0z^0+\Phi_\tau^0u=0.
\end{equation*}
Thus $u$ is admissible for the homogenized norm-optimal problem. The norm
of a dual Banach space is weak-* lower semicontinuous; see, for instance,
\cite[Chapter~III]{brezis}.  Therefore,
\begin{equation}\label{eq:liminf-N-epsilon}
N_0(\tau)\leqslant \|u\|_{L^\infty((0,\tau)\times I)}
\leqslant \liminf_{\varepsilon\to0}N_\varepsilon(\tau).
\end{equation}

\textit{Step 2. Proof of the upper bound.} Fix $\delta\in(0,\tau)$ and let
$u_0^{\tau-\delta}$ be the norm-optimal control for the homogenized system in
time $\tau-\delta$. Apply this same control to the $\varepsilon$-system on
$(0,\tau-\delta)$ and denote the resulting terminal state by $r_\varepsilon$, so that
\begin{equation*}\label{eq:residual-epsilon}
r_\varepsilon
=\mathbb T_{\tau-\delta}^\varepsilon z^0
+\Phi_{\tau-\delta}^\varepsilon u_0^{\tau-\delta}.
\end{equation*}
Since $u_0^{\tau-\delta}$ controls the homogenized system to zero in time
$\tau-\delta$, we have
\begin{equation*}
 \mathbb T_{\tau-\delta}^0z^0
 +\Phi_{\tau-\delta}^0u_0^{\tau-\delta}=0.
\end{equation*}
Subtracting this equality from the definition of $r_\varepsilon$ gives
\begin{equation}\label{eq:decomposition-residual-epsilon}
 r_\varepsilon
 =\left(\mathbb T_{\tau-\delta}^\varepsilon
        -\mathbb T_{\tau-\delta}^0\right)z^0
 +\left(\Phi_{\tau-\delta}^\varepsilon
        -\Phi_{\tau-\delta}^0\right)u_0^{\tau-\delta}.
\end{equation}
Here $\delta\in(0,\tau)$ is fixed. The first term on the right-hand side
of \eqref{eq:decomposition-residual-epsilon} converges strongly to zero
in $L^2(0,1)$ by Proposition
\ref{prop:semigroup-homogenization}. Moreover,
$u_0^{\tau-\delta}\in L^\infty((0,\tau-\delta)\times I)$, and hence it
belongs to $L^2((0,\tau-\delta)\times I)$ and is independent of
$\varepsilon$. The second term therefore converges strongly to zero by
Lemma \ref{lem:input-map-homogenization}. Consequently,
\begin{equation*}\label{eq:residual-converges-zero}
r_\varepsilon\longrightarrow0
\quad\text{in }L^2(0,1).
\end{equation*}
Apply Lemma \ref{lem:uniform-correction} to $r_\varepsilon$ and translate
the resulting control to the interval $(\tau-\delta,\tau)$. We obtain
$w_\varepsilon\in L^\infty((\tau-\delta,\tau)\times I)$ satisfying
\begin{equation}\label{eq:translated-correction-terminal-equation}
 \mathbb T_\delta^\varepsilon r_\varepsilon
 +\int_{\tau-\delta}^\tau
 \mathbb T_{\tau-t}^\varepsilon Bw_\varepsilon(t)\,{\rm d}t=0
\end{equation}
and
\begin{equation}\label{eq:small-correction-control}
 \|w_\varepsilon\|_{L^\infty((\tau-\delta,\tau)\times I)}
 \leqslant C_\delta\|r_\varepsilon\|_{L^2(0,1)}
 \longrightarrow0.
\end{equation}
Define the concatenated control
\begin{equation*}\label{eq:concatenated-recovery-control}
 v_\varepsilon(t,x)=
 \begin{cases}
 u_0^{\tau-\delta}(t,x),&0<t<\tau-\delta,\\
 w_\varepsilon(t,x),&\tau-\delta<t<\tau.
 \end{cases}
\end{equation*}
Using the semigroup property and the definition of $r_\varepsilon$, we
obtain
\begin{align*}
 \mathbb T_\tau^\varepsilon z^0+\Phi_\tau^\varepsilon v_\varepsilon
 &=\mathbb T_\delta^\varepsilon
 \left(\mathbb T_{\tau-\delta}^\varepsilon z^0
 +\Phi_{\tau-\delta}^\varepsilon u_0^{\tau-\delta}\right)
 +\int_{\tau-\delta}^\tau
 \mathbb T_{\tau-t}^\varepsilon Bw_\varepsilon(t)\,{\rm d}t\\
 &=0
\end{align*}
by \eqref{eq:translated-correction-terminal-equation}. Hence
$v_\varepsilon$ is admissible for the $\varepsilon$-problem in time
$\tau$. Since its two pieces have disjoint time supports,
\begin{equation*}\label{eq:norm-concatenated-recovery-control}
 \|v_\varepsilon\|_{L^\infty((0,\tau)\times I)}
 =\max\left\{
 N_0(\tau-\delta),
 \|w_\varepsilon\|_{L^\infty((\tau-\delta,\tau)\times I)}
 \right\}.
\end{equation*}
It follows from \eqref{eq:small-correction-control} that
\begin{equation}\label{eq:limsup-before-delta}
\limsup_{\varepsilon\to0}N_\varepsilon(\tau)
\leqslant N_0(\tau-\delta).
\end{equation}
No upper-semicontinuity theorem is used here: the limsup estimate follows
from the explicitly constructed admissible controls $v_\varepsilon$.
Finally, the continuity of $\tau\mapsto N_0(\tau)$ is part of Theorem
\ref{tm:normoptcon}, quoted from
\cite[Theorems~3.2 and~3.3]{tucsnak2025norm}. Letting $\delta\to0^+$ in
\eqref{eq:limsup-before-delta} yields
\begin{equation*}\label{eq:limsup-N-epsilon}
\limsup_{\varepsilon\to0}N_\varepsilon(\tau)\leqslant N_0(\tau).
\end{equation*}
Together with \eqref{eq:liminf-N-epsilon}, this proves
\eqref{eq:convergence-N-epsilon}.
\end{proof}

We are now in a position to prove our first main result.

\begin{proof}[Proof of Theorem~\ref{thm:introduction-norm-optimal}]
Existence, uniqueness and the bang-bang identity
\eqref{eq:bang-bang-norm-optimal-epsilon} follow from Proposition
\ref{prop:uniform-H1-H3} and Theorem \ref{tm:normoptcon}. Convergence of
the optimal values is Proposition
\ref{prop:convergence-norm-optimal-values}. It remains to prove the
asserted convergence of the optimal controls.

Every sequence $\varepsilon_n\to0$ has, by Proposition
\ref{prop:uniform-bound-N-epsilon}, a subsequence such that
$u_{\varepsilon_n}^\tau$ converges weakly-* to some $u$. The argument leading
to \eqref{eq:terminal-equation-limit} shows that $u$ controls the
homogenized system to zero in time $\tau$. Moreover, Proposition
\ref{prop:convergence-norm-optimal-values} and weak-* lower
semicontinuity of the $L^\infty$-norm, recalled in Step~1 above, give
\[
\|u\|_{L^\infty((0,\tau)\times I)}\leqslant N_0(\tau).
\]
Hence $u$ is norm-optimal. Its uniqueness, supplied by
\cite{tucsnak2025norm}, implies $u=u_0^\tau$. Since every subsequence has the
same possible limit, the entire family satisfies
\begin{equation*}\label{eq:weak-star-convergence-optimal-controls}
u_\varepsilon^\tau\xrightharpoonup{*}u_0^\tau
\quad\text{in }L^\infty((0,\tau)\times I).
\end{equation*}

The weak-* convergence implies weak convergence in
$L^2((0,\tau)\times I)$. By the bang-bang property
\eqref{eq:bang-bang-norm-optimal-epsilon} and Proposition
\ref{prop:convergence-norm-optimal-values},
\begin{equation*}\label{eq:L2-norms-optimal-controls}
\|u_\varepsilon^\tau\|_{L^2((0,\tau)\times I)}^2
=\tau|I|N_\varepsilon(\tau)^2
\longrightarrow \tau|I|N_0(\tau)^2
=\|u_0^\tau\|_{L^2((0,\tau)\times I)}^2.
\end{equation*}
The uniform convexity of $L^2$ therefore gives strong convergence in
$L^2$. Strong convergence in $L^p$ for $1\leqslant p<2$ follows from H\"older's
inequality, whereas the case $2<p<\infty$ follows by interpolation with
the uniform $L^\infty$ bound. Consequently,
\begin{equation*}\label{eq:strong-Lp-convergence-optimal-controls}
u_\varepsilon^\tau\longrightarrow u_0^\tau
\quad\text{strongly in }L^p((0,\tau)\times I)
\qquad(1\leqslant p<\infty),
\end{equation*}
which completes the proof.
\end{proof}

\begin{remark}\label{rem:fixed-space-argument}
The proof uses only the uniform estimate \eqref{eq:uniform-H3}, compactness
in the fixed control space, and the short-final-correction argument. This is
the same general principle that underlies the uniform null-control
construction in \cite{Lop_Zuaz}, although the spectral estimate of Section
\ref{sec:spectral-observability} allows us to avoid the three-step
construction used there.
\end{remark}
\section{Proof of Theorem~\ref{th_eps_fix_c}}
\label{sec:time-optimal-homogenization}

We now pass from norm-optimal controls to time-optimal controls. The argument
follows the general pattern of \cite{tucsnak2016perturbations}, but the main
compactness ingredient is supplied here by the convergence of norm-optimal
values proved in Proposition
\ref{prop:convergence-norm-optimal-values}. The
strict monotonicity, continuity, uniqueness and bang-bang results used below
are those recalled from \cite{tucsnak2025norm} in Section
\ref{sec:abstract-time-optimal}.

Fix $z^0\in L^2(0,1)\setminus\{0\}$ and $M>0$. For
$\varepsilon\geqslant 0$, define
\begin{equation*}\label{eq:time-optimal-value-epsilon}
\tau_\varepsilon^*(M)
=\inf\left\{\tau>0:\text{ there exists }u\in
L^\infty((0,\tau)\times I),\ \|u\|_{L^\infty}\leqslant M,
\ \mathbb T_\tau^\varepsilon z^0+\Phi_\tau^\varepsilon u=0\right\}.
\end{equation*}
The function $\tau_\varepsilon^*$ also depends on $z^0$, but we suppress this dependence to simplify the notation. 
For each $\varepsilon\geqslant 0$, the results of
\cite{tucsnak2025norm} and Proposition \ref{prop:uniform-H1-H3} imply that this infimum is attained by a unique
control $u_\varepsilon^*$ and that
\begin{equation}\label{eq:time-norm-inverse-relation}
N_\varepsilon(\tau_\varepsilon^*(M))=M.
\end{equation}
Moreover,
\begin{equation*}\label{eq:bang-bang-time-optimal-epsilon}
|u_\varepsilon^*(t,x)|=M
\quad\text{for almost every }(t,x)\in
(0,\tau_\varepsilon^*(M))\times I.
\end{equation*}

We first record a uniform upper bound for the optimal times.

\begin{proposition}\label{prop:uniform-upper-optimal-times}
There exists $\widehat \tau>0$, depending only on $M$, $K$, $|I|$ and
$\|z^0\|_{L^2(0,1)}$, such that
\begin{equation}\label{eq:uniform-upper-optimal-times}
\tau_\varepsilon^*(M)\leqslant \widehat \tau
\qquad\qquad(\varepsilon\geqslant 0).
\end{equation}
\end{proposition}

\begin{proof}
The Poincar\'e inequality and \eqref{low_coef_c} implies that, if we denote by $\lambda_1(A_\varepsilon)$ the smallest eigenvalue  of $A_\varepsilon$, that
\begin{equation*}\label{eq:uniform-first-eigenvalue}
\inf_{\varepsilon\geqslant 0}\lambda_1(A_\varepsilon)
\geqslant\frac{\pi^2}{K}.
\end{equation*}
Let $\tau>1$. Allow the system to evolve freely on $(0,\tau-1)$ and use the
uniform null-control estimate on the final interval $(\tau-1,\tau)$. This gives
\begin{equation*}\label{eq:uniform-decay-norm-optimal-value}
N_\varepsilon(\tau)
\leqslant C_1\exp\left(-\frac{\pi^2}{K}(\tau-1)\right)
\|z^0\|_{L^2(0,1)},
\end{equation*}
where $C_1$ is independent of $\varepsilon$. Choose $\tau$ so that
the right-hand side is at most $M$,
and fix $\widehat\tau=\tau+1$. Then
$N_\varepsilon(\widehat \tau)\leqslant M$. Since $N_\varepsilon$ is strictly
decreasing and satisfies \eqref{eq:time-norm-inverse-relation}, we obtain
\eqref{eq:uniform-upper-optimal-times}.
\end{proof}

A difficulty raised by the proof of our second main result is that the passage to the limit has to be done for functions which are not necessarily defined on the same time interval.
The following consequence of the results in Section \ref{sec:operator-approximation} will
be useful to tackle this issue.

\begin{lemma}\label{lem:varying-time-terminal-limit}
Let $\tau_\varepsilon\to \tau>0$. Then, for all $z^0\in L^2(0,1)$:
\begin{equation}\label{eq:varying-semigroup-time-limit}
\mathbb T_{\tau_\varepsilon}^\varepsilon z^0
\longrightarrow\mathbb T_\tau^0z^0
\quad\text{in }L^2(0,1).
\end{equation}
Moreover, let $(u_\varepsilon)$ be bounded in
$L^\infty((0,\widehat \tau)\times I)$, where
$\tau_\varepsilon,\tau<\widehat \tau$. Suppose that
\begin{equation*}\label{eq:support-varying-controls}
u_\varepsilon=0\quad\text{almost everywhere on }
(\tau_\varepsilon,\widehat \tau)\times I
\end{equation*}
and
\begin{equation*}\label{eq:weak-star-varying-controls}
u_\varepsilon\xrightharpoonup{*}u
\quad\text{in }L^\infty((0,\widehat \tau)\times I).
\end{equation*}
Then $u=0$ almost everywhere on $(\tau,\widehat \tau)\times I$ and
\begin{equation}\label{eq:varying-input-map-limit}
\int_0^{\tau_\varepsilon}
\mathbb T_{\tau_\varepsilon-t}^\varepsilon Bu_\varepsilon(t)\,{\rm d}t
\rightharpoonup
\int_0^\tau\mathbb T_{\tau-t}^0Bu(t)\,{\rm d}t
\quad\text{in }L^2(0,1).
\end{equation}

\end{lemma}

\begin{proof}
Let us first prove \eqref{eq:varying-semigroup-time-limit}. 
We have that:
\[\mathbb T_{\tau_\varepsilon}^\varepsilon z^0=
\mathbb T^0_{\tau_\varepsilon}z^0
+(\mathbb T_{\tau_\varepsilon}^\varepsilon z^0-\mathbb T^0_{\tau_\varepsilon}z^0).
\]
By the continuity of the semi-group we have that:
\[\mathbb T_{\tau_\varepsilon}^0 z^0\to \mathbb T_{\tau}^0 z^0.
\]
Moreover, by Proposition \ref{prop:semigroup-homogenization}, 
using the uniform convergence in $C^0([0,\tau+1];L^2(0,1))$:
\[\mathbb T_{\tau_\varepsilon}^\varepsilon z^0-\mathbb T_{\tau_\varepsilon}^0 z^0\to 0.
\]
Thus, we obtain \eqref{eq:varying-semigroup-time-limit}.

As for the assertion concerning the support follows directly from
\eqref{eq:support-varying-controls} and weak-* convergence. Finally, to prove
\eqref{eq:varying-input-map-limit}, take $f\in L^2(0,1)$ and extend every
control by zero to $(0,\widehat \tau)$. By self-adjointness, the relevant
duality product is
\begin{equation}\label{eq:varying-time-duality-product}
\int_0^{\widehat \tau}\int_Iu_\varepsilon(t,x)
\overline{\chi_{(0,\tau_\varepsilon)}(t)
(\mathbb T_{\tau_\varepsilon-t}^\varepsilon f)(x)}
\,{\rm d}x\,{\rm d}t.
\end{equation}
We have that:
\begin{equation*}\label{eq:varying-time-kernel-convergence}
\chi_{(0,\tau_\varepsilon)}(t)
\mathbb T_{\tau_\varepsilon-t}^\varepsilon f|_I
\longrightarrow
\chi_{(0,\tau)}(t)\mathbb T_{\tau-t}^0f|_I
\quad\text{in }L^1((0,\widehat \tau)\times I).
\end{equation*}
Indeed, the semigroup convergence is uniform on compact time intervals by Proposition \ref{prop:semigroup-homogenization},
the limit semigroup is strongly continuous, and the symmetric difference
of $(0,\tau_\varepsilon)$ and $(0,\tau)$ has measure tending to zero. We may
therefore pass to the limit in
\eqref{eq:varying-time-duality-product}, proving
\eqref{eq:varying-input-map-limit}. 
\end{proof}

\begin{proof}[Proof of Theorem~\ref{th_eps_fix_c}]
Existence, uniqueness and the bang-bang property of the time-optimal
controls follow from Proposition \ref{prop:uniform-H1-H3} and Theorem
\ref{tm:timeoptcon}. The uniform bound for the optimal times is
Proposition \ref{prop:uniform-upper-optimal-times}. We now prove the
convergence assertions.

We first prove convergence of the optimal times. Let
$\delta\in(0,\tau_0^*(M))$. Since $N_0$ is continuous and strictly
decreasing and
\begin{equation*}\label{eq:identity-limit-time-norm}
N_0(\tau_0^*(M))=M,
\end{equation*}
we have
\begin{equation}\label{eq:strict-bracketing-N0}
N_0(\tau_0^*(M)-\delta)>M>
N_0(\tau_0^*(M)+\delta).
\end{equation}
Proposition \ref{prop:convergence-norm-optimal-values}, applied at the two fixed
times in \eqref{eq:strict-bracketing-N0}, shows that, for all sufficiently
small $\varepsilon$,
\begin{equation*}\label{eq:strict-bracketing-Nepsilon}
N_\varepsilon(\tau_0^*(M)-\delta)>M>
N_\varepsilon(\tau_0^*(M)+\delta).
\end{equation*}
Using the strict decrease of $N_\varepsilon$ and
\eqref{eq:time-norm-inverse-relation}, we obtain
\begin{equation*}\label{eq:bracketing-optimal-times}
\tau_0^*(M)-\delta<\tau_\varepsilon^*(M)
<\tau_0^*(M)+\delta.
\end{equation*}
Since $\delta$ is arbitrary, we have proved
\begin{equation*}\label{eq:convergence-time-optimal-values}
\tau_\varepsilon^*(M)\longrightarrow\tau_0^*(M)
\qquad(\varepsilon\to0).
\end{equation*}

Choose $\widehat\tau$ as in Proposition
\ref{prop:uniform-upper-optimal-times}, increasing it if necessary, and
extend the time-optimal controls by zero:
\begin{equation}\label{eq:extended-time-optimal-controls}
\widetilde u_\varepsilon^*(t,x)
=\begin{cases}
u_\varepsilon^*(t,x),&0<t<\tau_\varepsilon^*(M),\\
0,&\tau_\varepsilon^*(M)\leqslant t<\widehat\tau.
\end{cases}
\end{equation}

The extensions \eqref{eq:extended-time-optimal-controls} are bounded by
$M$ in $L^\infty((0,\widehat \tau)\times I)$. Let a subsequence converge
weakly-* to some $u$. Lemma \ref{lem:varying-time-terminal-limit}, applied
with $\tau_\varepsilon=\tau_\varepsilon^*(M)$, allows us to pass to the limit
in
\begin{equation*}\label{eq:time-optimal-terminal-equation-epsilon}
\mathbb T_{\tau_\varepsilon^*(M)}^\varepsilon z^0
+\int_0^{\tau_\varepsilon^*(M)}
\mathbb T_{\tau_\varepsilon^*(M)-t}^\varepsilon
B u_\varepsilon^*(t)\,{\rm d}t=0.
\end{equation*}
It follows that $u=0$ after $\tau_0^*(M)$ and
\begin{equation*}\label{eq:time-optimal-terminal-equation-limit}
\mathbb T_{\tau_0^*(M)}^0z^0
+\int_0^{\tau_0^*(M)}
\mathbb T_{\tau_0^*(M)-t}^0Bu(t)\,{\rm d}t=0.
\end{equation*}
Moreover, $\|u\|_{L^\infty}\leqslant M$. Thus the restriction of $u$ to
$(0,\tau_0^*(M))\times I$ is a time-optimal control for the homogenized
system. By uniqueness, it is $u_0^*$. Hence every weak-* convergent
subsequence has the same limit, and therefore
\begin{equation*}\label{eq:weak-star-time-optimal-controls}
\widetilde u_\varepsilon^*
\xrightharpoonup{*}\widetilde u_0^*
\quad\text{in }L^\infty((0,\widehat\tau)\times I).
\end{equation*}

Finally, the bang-bang property gives
\begin{equation*}\label{eq:L2-norm-time-optimal-controls}
\|\widetilde u_\varepsilon^*\|_{L^2((0,\widehat \tau)\times I)}^2
=M^2|I|\tau_\varepsilon^*(M)
\longrightarrow
M^2|I|\tau_0^*(M)
=\|\widetilde u_0^*\|_{L^2((0,\widehat \tau)\times I)}^2.
\end{equation*}
Together with weak convergence in $L^2$, this implies strong convergence
in $L^2$. H\"older's inequality and interpolation with the uniform
$L^\infty$ bound then give
\begin{equation}\label{eq:strong-Lp-time-optimal-controls}
\widetilde u_\varepsilon^*
\longrightarrow\widetilde u_0^*
\quad\text{strongly in }L^p((0,\widehat\tau)\times I)
\qquad(1\leqslant p<\infty),
\end{equation}
which completes the proof.
\end{proof}

The result below gives, in particular, the convergence of the optimal state trajectories.

\begin{corollary}\label{cor:convergence-time-optimal-trajectories}
Let
\begin{equation*}\label{eq:time-optimal-trajectories}
z_\varepsilon^*(t)
=\mathbb T_t^\varepsilon z^0
+\int_0^t\mathbb T_{t-s}^\varepsilon
B\widetilde u_\varepsilon^*(s)\,{\rm d}s.
\end{equation*}
For every $S<\tau_0^*(M)$,
\begin{equation}\label{eq:trajectory-convergence-before-optimal-time}
z_\varepsilon^*\longrightarrow z_0^*
\quad\text{in }C([0,S];L^2(0,1)).
\end{equation}
In addition,
\begin{equation}\label{eq:terminal-states-zero}
z_\varepsilon^*(\tau_\varepsilon^*(M))=0,
\qquad
z_0^*(\tau_0^*(M))=0.
\end{equation}
\end{corollary}

\begin{proof}
For sufficiently small $\varepsilon$, one has
$S<\tau_\varepsilon^*(M)$. The semigroup term in
\eqref{eq:time-optimal-trajectories} converges uniformly on $[0,S]$ by
Proposition \ref{prop:semigroup-homogenization}. The control terms converge
uniformly because of the strong $L^2$ convergence in
\eqref{eq:strong-Lp-time-optimal-controls}, the uniform boundedness of the
input maps, and the strong convergence for a fixed input proved in Lemma
\ref{lem:input-map-homogenization}. This proves
\eqref{eq:trajectory-convergence-before-optimal-time}. The identities
\eqref{eq:terminal-states-zero} follow from the definition of the optimal
controls.
\end{proof}

\appendix

\section{Appendix: An abstract final state observability result}

In this appendix we provide an abstract version of a classical approach to final-state observability from
sets of positive measure. This type of result appears often in the control theoretic literature, but, as far as we know,
no abstract framework has been explicitly formulated. Consequently, with no claim of originality, we formulate here an abstract result in this direction and we provide its detailed proof.

Let $d\in \mathbb{N}$ and let $\Omega$ be an open bounded set in $\mathbb{R}^d$. Denote $X=L^2(\Omega)$  and let $A:\mathcal{D}(A)\to X$ be an unbounded positive
self-adjoint operator on $X$. We assume that $A$ has compact resolvents, so that
$-A$ generates an analytic semigroup $\mathbb{T}=(\mathbb{T})_{t\geqslant 0}$ on $X$.
We denote by
$(e_k)_{k\geqslant 1}$ an orthonormal basis of eigenvectors of $A$ and by
$(\lambda_k)_{k\geqslant 1}$ the corresponding eigenvalues, so that:
\begin{equation}\label{eq:abstract-eigenvalues}
Ae_k=\lambda_k e_k,\qquad
0<\lambda_1 \leqslant \lambda_2 \leqslant\cdots,\qquad
\lim_{k\to\infty}\lambda_k=+\infty.
\end{equation}

\begin{definition}\label{def:abstract-spectral-inequality}
We say that $A$ satisfies the measurable-set spectral inequality on $\mathcal{O}$ if, for
every $c>0$, there exists $N_c\geqslant 1$ such that
\begin{equation}\label{eq:abstract-spectral-inequality}
\left(\sum_{\sqrt{\lambda_k}\leqslant\mu}|b_k|^2\right)^{1/2}
\leqslant N_c\exp(N_c\mu)
\left\|\sum_{\sqrt{\lambda_k}\leqslant\mu}b_k e_k\right\|_{L^1(\mathcal O)}
\end{equation}
for every $\mu\geqslant 1$, every  sequence $(b_k)\in l^2(\mathbb{C})$, and every
measurable set $\mathcal O\subset\Omega$ such that  $|\mathcal O|\geqslant c$.
\end{definition}

\begin{theorem}\label{theo:poate_abs}
Assume that $A$ satisfies the measurable-set spectral inequality on $\mathcal{O}$. Let $\tau>0$ and let
$\mathcal D\subset(0,\tau)\times\Omega$ be measurable, with %positive $d+1$ dimensional Lebesgue measure
$|\mathcal D|>0$. Then there exists $K_{\tau,\mathcal D}>0$ such that
\begin{equation}\label{eq:abstract-final-observability}
\|\mathbb T_\tau\tau f\|_X
\leqslant K_{\tau,\mathcal D}
\int_{\mathcal D}|(\mathbb T_t f)(x)|\,{\rm d}x\,{\rm d}t
\qquad\qquad(f\in X).
\end{equation}
The constant depends on $A$ only through the constants $N_c$ in
\eqref{eq:abstract-spectral-inequality}.
\end{theorem}

\begin{proof}
\textit{Step 1. Low/high frequency decomposition and preliminary observability estimate.}
For $\mu\geqslant 1$, let $P_\mu$ be the orthogonal projection onto the span of
the eigenvectors for which the corresponding eigenvalues satisfy $\sqrt{\lambda_k}\leqslant \mu$. If $0\leqslant s<t$,
then
\begin{equation}\label{eq:abstract-high-frequency-decay}
\|(I-P_\mu)\mathbb T_t f\|_X
\leqslant \exp\bigl(-\mu^2(t-s)\bigr)
\|\mathbb T_s f\|_X.
\end{equation}
Fix a measurable set $\mathcal O\subset\Omega$ with $|\mathcal O|\geqslant c$.
Applying \eqref{eq:abstract-spectral-inequality} to
$P_\mu\mathbb T_t f$, and estimating the high-frequency part by
\eqref{eq:abstract-high-frequency-decay}, gives
\begin{equation*}
\begin{split}
\|\mathbb T_t f\|_X &\leqslant N_c\exp(N_c\mu) \|P_\mu\mathbb T_t f\|_{L^1(\mathcal O)} +
\exp\bigl(-\mu^2(t-s)\bigr)\|\mathbb T_s f\|_X \\
&\leqslant N_c\exp(N_c\mu)\|\mathbb T_t f\|_{L^1(\mathcal O)}+
N_c\exp(N_c\mu) \|P_\mu\mathbb T_t f-\mathbb T_t f\|_{L^1(\mathcal O)}
+\exp\bigl(-\mu^2(t-s)\bigr)\|\mathbb T_s f\|_X\\
&\leqslant N_c\exp(N_c\mu)\|\mathbb T_t f\|_{L^1(\mathcal O)}+
N_c|\mathcal O|^{1/2}\exp(N_c\mu) \|P_\mu\mathbb T_t f-\mathbb T_t f\|_{X}
+\exp\bigl(-\mu^2(t-s)\bigr)\|\mathbb T_s f\|_X
\\&\leqslant N_c\exp(N_c\mu)\|\mathbb T_t f\|_{L^1(\mathcal O)}+
N_c|\mathcal O|^{1/2}\exp\bigl(N_c\mu-\mu^2(t-s)\bigr)\|\mathbb T_s f\|_{X}
+\exp\bigl(-\mu^2(t-s)\bigr)\|\mathbb T_s f\|_X
\\
&\leqslant N_c\exp(N_c\mu)\|\mathbb T_t f\|_{L^1(\mathcal O)}
+(N_c|\mathcal O|^{1/2}+1)\exp\bigl(N_c\mu-\mu^2(t-s)\bigr)
\|\mathbb T_s f\|_X.
\end{split}
\end{equation*}
In conclusion, readjusting the value of $N_c$ we have that:
\begin{equation}\label{eq:perpestX}
\|\mathbb T_t f\|_X\leqslant  N_ce^{N_c\mu}\left[\|\mathbb T_tf\|_{L^1(\mathcal O)}+e^{-\mu^2(t-s)}\|\mathbb T_s f\|_{X}\right].
\end{equation}

\textit{Step 2. One-point-in-time interpolation inequality.}
Let us now prove that there are
constants $C_c\geqslant 1$ and $\theta_c\in(0,1)$ such that for every $f\in X$, $t>0$ and $s\in (0,t)$ we have
\begin{equation}\label{est:interpnorm}
\|\mathbb T_t f\|_X
\leqslant
\left(C_c\exp\left(\frac{C_c}{t-s}\right)
\|\mathbb T_t f\|_{L^1(\mathcal O)}\right)^{\theta_c}
\|\mathbb T_s f\|_X^{1-\theta_c}.
\end{equation}

For that, we remark that for each $N_c>0$ there exists $\tilde N_c>0$ such that:
\begin{equation}\label{est:eqNmumu2}
N_c\mu-\mu^2\frac{N_c-1}{N_c}(t-s)\leqslant \frac{\tilde N_c}{t-s}  \qquad\qquad(t>s>0).
\end{equation}
Indeed, the maximum of 
\begin{equation}\label{eq:mumapstoNmu-mu2}
\mu\mapsto N_c\mu-\mu^2\frac{N_c-1}{N_c}(t-s) 
\end{equation}
is obtained when
\[\mu=\frac{N_c^2}{2(N_c-1)}\cdot\frac{1}{t-s},\]
so by replacing this value in \eqref{eq:mumapstoNmu-mu2} we obtain \eqref{est:eqNmumu2} with $\tilde N_c=\frac{N_c^3}{4(N_c-1)}$.

It follows from \eqref{est:eqNmumu2} that
\[e^{N_c\mu}\leqslant e^{\mu^2\frac{N_c-1}{N_c}(t-s)}e^{\frac{\tilde N_c}{t-s}}.\]
So, combining this with \eqref{eq:perpestX}, we obtain that:
\begin{equation}\label{est:eN24t-s}
\begin{split}
\|\mathbb T_t  f\|_X&\leqslant N_ce^{\frac{\tilde N_c}{t-s}}\left[e^{\mu^2\frac{N_c-1}{N_c}(t-s)}\|\mathbb T_tf\|_{L^1(\mathcal O)} + e^{-\mu^2(t-s)/N_c} \|\mathbb T_sf\|_X\right]\\&
\leqslant N_ce^{\frac{\tilde N_c}{t-s}}\left[e^{\mu^2(t-s)}\|\mathbb T_tf\|_{L^1(\mathcal O)} + e^{-\mu^2(t-s)/N_c} \|\mathbb T_sf\|_X\right].
\end{split}
\end{equation}
It is important to remark that \eqref{est:eN24t-s} is true for all $\mu\geqslant0$. Notably, let us consider $\epsilon=e^{-\mu^2(t-s)/N_c}$. By choosing the appropriate value for $\mu$, %and reassigning the value of $N_c$, 
we have that the following estimate is true:
\[\|\mathbb T_tf\|_{X}\leqslant N_ce^{\frac{\tilde N_c}{t-s}}\left[\epsilon^{-N_c}\|\mathbb T_tf\|_{L^1(\mathcal O)}+\epsilon \|\mathbb T_sf\|_X.\right] \qquad \forall \epsilon\in(0,1].
\]
If $\|\mathbb T_tf\|_{L^1(\mathcal O)}=0$, by taking $\epsilon\to 0$, we obtain that $\|\mathbb T_tf\|_{X}=0$, and in particular \eqref{est:interpnorm} is true. Thus, from now on we suppose that $\|\mathbb T_tf\|_{L^1(\mathcal O)}\neq0$.

By optimizing the function $\epsilon\mapsto \epsilon^{-N_c}\|\mathbb T_tf\|_{L^1(\mathcal O)}+\epsilon \|\mathbb T_sf\|_X$ in $[0,\infty)$, we have that the global minimum is in:
\[\epsilon=\left( \frac{N_c\|\mathbb T_tf\|_{L^1(\mathcal O)}}{\|\mathbb T_sf\|_X}\right)^{\frac{1}{N_c+1}}.
\]
We have two options:
\begin{itemize}
\item If $\epsilon\leqslant 1$, we pick that value, and obtain:
\[\|\mathbb T_tf\|_{X}\leqslant N_c\left[N_c^{-N_c/(N_c+1)}+N_c^{1/(N_c+1)}\right]
e^{\frac{\tilde N_c}{t-s}}\|\mathbb T_tf\|_{L^1(\mathcal O)}^{1/(N_c+1)}\|\mathbb T_sf\|_X^{N_c/(N_c+1)}.
\]
\item Otherwise, we have that:
$N_c\|\mathbb T_tf\|_{L^1(\mathcal O)}\geqslant \|\mathbb T_sf\|_X$. Thus, in this case the estimate follows from the decay estimate, as, for any $\theta\in(0,1)$:
\[\|\mathbb T_tf\|_X\leqslant \|\mathbb T_sf\|_X
=\|\mathbb T_sf\|_X^\theta\|\mathbb T_sf\|_X^{1-\theta}
\leqslant N_c^\theta\|\mathbb T_tf\|^\theta_{L^1(\mathcal O)}
\|\mathbb T_sf\|_X^{1-\theta}.\]
\end{itemize}
In both cases, \eqref{est:interpnorm} follows.

\textit{Step 3. Space-time interpolation inequality on an interval $(t_1,t_2)$.}
For almost every $t\in(0,\tau)$, set
\begin{equation*}\label{eq:sections-of-D}
\mathcal D_t=\{x\in\Omega:(t,x)\in\mathcal D\}
\end{equation*}
and define
\begin{equation*}\label{eq:good-time-set}
E=\left\{t\in(0,\tau):\mathcal D_t \mbox{ is measurable and }|\mathcal D_t|\geqslant
\frac{|\mathcal D|}{2\tau}\right\}.
\end{equation*}
Fubini's theorem gives
\begin{equation*}\label{eq:measure-good-time-set}
|E|\geqslant \frac{|\mathcal D|}{2|\Omega|},
\end{equation*}
because
\[
|\mathcal D|=\int_0^\tau|\mathcal D_t|\,{\rm d}t
\leqslant |\Omega||E|+\frac{|\mathcal D|}{2}.
\]
Thus \eqref{est:interpnorm} applies at every 
$t\in E$, with $\mathcal O=\mathcal D_t$ and constants depending only on
$\frac{|\mathcal D|}{2\tau}$ and $|\Omega|$.

Let $0\leqslant t_1<t_2\leqslant \tau$ and assume that
$|E\cap(t_1,t_2)|\geqslant\eta(t_2-t_1)$ for some $\eta\in(0,1)$. Let 
$$t_3=t_1+\frac{\eta}{2}(t_2-t_1).$$
Then,
\begin{equation}\label{est:t3t2E}
    |E\cap (t_3,t_2)|=|E\cap (t_1,t_2)|-|E\cap (t_1,t_3)|\geqslant \eta(t_2-t_1) - \frac{\eta}{2}(t_2-t_1)=\frac{\eta}{2}(t_2-t_1).
\end{equation}
Integrating
\eqref{est:interpnorm} on $(t_3,t_2)\cap E$
and using H\"older's inequality gives
constants $C\geqslant 1$ and $\theta\in(0,1)$ such that
\begin{equation*}
|E\cap (t_3,t_2)| \|\mathbb T_{t_2}f\|_X
\leqslant \left(C\exp\left(\frac{C}{t_2-t_3}\right)
\int_{t_3}^{t_2}\chi_E(t)
\|\mathbb T_t f\|_{L^1(\mathcal D_t)}\,{\rm d}t\right)^\theta\|\mathbb T_{t_1}f\|_X^{1-\theta}.
\end{equation*}
Thus, using \eqref{est:t3t2E} we obtain that:
\begin{equation*} \|\mathbb T_{t_2}f\|_X
\leqslant \frac{2}{\mu(t_2-t_1)}
\left(C\exp\left(\frac{C}{(1-\eta/2)(t_2-t_1)}\right)
\int_{t_3}^{t_2}\chi_E(t)
\|\mathbb T_t f\|_{L^1(\mathcal D_t)}\,{\rm d}t\right)^\theta\|\mathbb T_{t_1}f\|_X^{1-\theta}.
\end{equation*}
Consequently, using that $1/s\leqslant e^{C/s}$ for some $C>0$ and all $s>0$, and reassigning the value of $C$ if needed, we obtain that: 
\begin{equation}\label{eq:estthetaDs}
\begin{split}
\|\mathbb T_{t_2}f\|_X
&\leqslant \left(C\exp\left(\frac{C}{t_2-t_1}\right)
\int_{t_1}^{t_2}\chi_E(t)
\|\mathbb T_t f\|_{L^1(\mathcal D_t)}\,{\rm d}t\right)^\theta\|\mathbb T_{t_1}f\|_X^{1-\theta}.
\end{split}
\end{equation}

\textit{Step 4. Conclusion with a telescoping sum.}
Multiplying the right-hand side of \eqref{eq:estthetaDs} by $(\epsilon^{-(1-\theta)})^\theta (\epsilon^\theta)^{1-\theta}$, and using Young's inequality:
\[\|\mathbb T_{t_2}f\|_X\leqslant \epsilon^{-(1-\theta)}Ce^{C/(t_2-t_1)}\int_{t_1}^{t_2}\chi_E(s)\|\mathbb T_tf\|_{L^1(\mathcal D_t)}dt+\epsilon^\theta\|\mathbb T_ {t_1}f\|_X,
\]
for all $\epsilon>0$. This implies that:
\[\epsilon^{1-\theta} e^{-C/(t_2-t_1)}\|\mathbb T_{t_2}f\|_X -\epsilon e^{-C/(t_2-t_1)}\|\mathbb T_ {t_1}f\|_X\leqslant C\int_{t_1}^{t_2}\chi_E(s)\|\mathbb T_sf\|_{L^1(\mathcal D_s)}ds.
\]
In particular, choosing $\epsilon=e^{-\frac{1}{t_2-t_1}}$ leads to:
\begin{equation}\label{eq:abstract-telescoping-form}
\begin{split}
&\exp\left(-\frac{C+1-\theta}{t_2-t_1}\right)
\|\mathbb T_{t_2}f\|_{L^2(\Omega)}
-\exp\left(-\frac{C+1}{t_2-t_1}\right)
\|\mathbb T_{t_1}f\|_{L^2(\Omega)}\\
&\qquad\leqslant C\int_{t_1}^{t_2}\chi_E(t)
\|\mathbb T_t f\|_{L^1(\mathcal D_t)}\,{\rm d}t.
\end{split}
\end{equation}
Choose a density point $\ell$ of $E$; that is, a point such that $\lim_{r\to0^+}\frac{|E\cap B(\ell,r)|}{|B(\ell,r)|}=1$, for $B$ the open Euclidean ball of radious $r$ centered in $\ell$. The standard density-point lemma (see \cite[pp. 256-257]{lions1971optimal})
provides $\ell_1\in(\ell,\tau)$ and a sequence
\begin{equation*}\label{eq:density-sequence}
\ell_{j+1}=\ell+q^{-j}(\ell_1-\ell),
\qquad
|E\cap(\ell_{j+1},\ell_j)|\geqslant
\frac13(\ell_j-\ell_{j+1}) \qquad\qquad(j\in \mathbb{N}).
\end{equation*}
Taking $q=(C+1)/(C+1-\theta)$, applying
\eqref{eq:abstract-telescoping-form} on every
$(\ell_{j+1},\ell_j)$, and summing makes the left-hand sides telescope.
We obtain
\[
\|\mathbb T_{\ell_1}f\|_X
\leqslant C_{\tau,\mathcal D}
\int_{\mathcal D}|(\mathbb T_t f)(x)|\,{\rm d}x\,{\rm d}t.
\]
Since $\mathbb T$ is a contraction semigroup on $X$ and
$\ell_1<\tau$, this implies \eqref{eq:abstract-final-observability}.

\end{proof}

\bibliographystyle{plain}
\bibliography{references}

\begin{thebibliography}{10}

\bibitem{alessandrini2008null}
G.~Alessandrini and L.~Escauriaza.
\newblock Null-controllability of one-dimensional parabolic equations.
\newblock {\em ESAIM COCV}, 14(2):284--293, 2008.

\bibitem{apraiz2013null}
J.~Apraiz and L.~Escauriaza.
\newblock Null-control and measurable sets.
\newblock {\em ESAIM COCV}, 19(1):239--254, 2013.

\bibitem{AvellanedaBardosRauch}
M.~Avellaneda, C.~Bardos, and J.~Rauch.
\newblock Contr\^olabilit\'e exacte, homog\'en\'eisation et localisation d'ondes dans un milieu non homog\`ene.
\newblock {\em Asymptotic Anal.}, 5:481--494, 1992.

\bibitem{Axler2020}
S.~Axler.
\newblock {\em Measure, Integration \& Real Analysis}, volume 282 of {\em Graduate Texts in Mathematics}.
\newblock Springer, Cham, 2020.

\bibitem{BonifaciusPieper2019}
L.~Bonifacius and K.~Pieper.
\newblock Strong stability of linear parabolic time-optimal control problems.
\newblock {\em ESAIM COCV}, 25:1, 2019.

\bibitem{brezis}
H.~Brezis.
\newblock {\em Functional analysis, Sobolev spaces and partial differential equations}.
\newblock Universitext. Springer, New York, 2011.

\bibitem{Castro1999}
C.~Castro.
\newblock Boundary controllability of the one-dimensional wave equation with periodic oscillating density.
\newblock {\em Asymptotic Anal.}, 20(3--4):317--350, 1999.

\bibitem{CastroZuazua1997}
C.~Castro and E.~Zuazua.
\newblock Contr\^ole de l'{\'e}quation des ondes \`a densit{\'e} rapidement oscillante \`a une dimension d'espace.
\newblock {\em C.R. Math.}, 324(11):1237--1242, 1997.

\bibitem{CastroZuazua2000}
C.~Castro and E.~Zuazua.
\newblock Low frequency asymptotic analysis of a string with rapidly oscillating density.
\newblock {\em SIAM J. on Appl. Math.}, 60(4):1205--1233, 2000.

\bibitem{KunischWang2013}
K.~Kunisch and L.~Wang.
\newblock Time optimal control of the heat equation with pointwise control constraints.
\newblock {\em ESAIM COCV}, 19(2):460--485, 2013.

\bibitem{LinShen2022}
F.~Lin and Z.~Shen.
\newblock Uniform boundary controllability and homogenization of wave equations.
\newblock {\em J. Eur. Math. Soc.}, 24(9):3031--3053, 2022.

\bibitem{lions1971optimal}
J.-L. Lions.
\newblock {\em Optimal control of systems governed by partial differential equations}, volume 170.
\newblock Springer, 1971.

\bibitem{Lop_Zuaz}
A.~L{\'o}pez and E.~Zuazua.
\newblock Uniform null-controllability for the one-dimensional heat equation with rapidly oscillating periodic density.
\newblock {\em Ann. Inst. H. Poincar\'e Anal. Non Lin\'eaire}, 19(5):543--580, 2002.

\bibitem{MeshkovaSuslina2016}
Yu.~M. Meshkova and T.~A. Suslina.
\newblock Homogenization of initial boundary value problems for parabolic systems with periodic coefficients.
\newblock {\em Appl. Anal.}, 95(8):1736--1775, 2016.

\bibitem{MRT}
S.~Micu, I.~Roventa, and M.~Tucsnak.
\newblock Time optimal boundary controls for the heat equation.
\newblock {\em J. Funct. Anal.}, 263(1):25--49, 2012.

\bibitem{micu2012time}
S.~Micu, I.~Roventa, and M.~Tucsnak.
\newblock Time optimal boundary controls for the heat equation.
\newblock {\em J. Func. Anal.}, 263(1):25--49, 2012.

\bibitem{Pontryagin}
L.~S. Pontryagin, V.~G. Boltyanskii, R.~V. Gamkrelidze, and E.~F. Mishchenko.
\newblock {\em The Mathematical Theory of Optimal Processes}.
\newblock Interscience, 1962.

\bibitem{tucsnak2025norm}
M.~Tucsnak.
\newblock On norm and time optimal controls for systems described by linear parabolic pdes.
\newblock {\em SIAM J. Control Optim.}, 63(1):349--374, 2025.

\bibitem{tucsnak2016perturbations}
M.~Tucsnak, G.~Wang, and C.-T. Wu.
\newblock Perturbations of time optimal control problems for a class of abstract parabolic systems.
\newblock {\em SIAM J. Control Optim.}, 54(6):2965--2991, 2016.

\bibitem{wang_linf}
G.~Wang.
\newblock {$\mathcal{L^\infty}$}-null controllability for the heat equation and its consequences for the time optimal control problem.
\newblock {\em SIAM J. Control Optim.}, 47(4):1701--1720, 2008.

\bibitem{WangEtAl2018}
G.~Wang, L.~Wang, Y.~Xu, and Y.~Zhang.
\newblock {\em Time Optimal Control of Evolution Equations}, volume~92 of {\em Progress in Nonlinear Differential Equations and Their Applications}.
\newblock Birkh\"auser, Cham, 2018.

\bibitem{WangZuazua}
G.~Wang and E.~Zuazua.
\newblock On the equivalence of minimal time and minimal norm controls for internally controlled heat equations.
\newblock {\em SIAM J. Control Optim.}, 50(5):2938--2958, 2012.

\bibitem{Yu}
H.~Yu.
\newblock Approximation of time optimal controls for heat equations with perturbations in the system potential.
\newblock {\em SIAM J. Control Optim.}, 52(3):1663--1692, 2014.

\bibitem{zhu2025propagation}
Y.~Zhu.
\newblock Propagation of smallness for solutions of elliptic equations in the plane.
\newblock {\em Math. Eng.}, 7(1):1--12, 2025.

\end{thebibliography}

\end{document}